\documentclass[11pt]{article}
\usepackage{amsmath,amsthm,amssymb,amsfonts}
\usepackage{float}
\usepackage{color,fullpage}
\usepackage{cite}
\usepackage[title,titletoc]{appendix}
\usepackage{algorithm,algorithmic}
\usepackage{tikz}
\usepackage[overload]{empheq}

\usepackage{hyperref}
\usepackage{cleveref}
\hypersetup{hidelinks} %remote the colorful box on the hyperlink
\usepackage{amsthm}
\usepackage{lipsum}

\newcommand{\RR}{\mathbb{R}}

\newcommand{\prox}{{\mathbf{prox}}}

\newcommand{\cS}{{\mathcal{S}}}

\newcommand{\cC}{{\mathcal{C}}}

\newtheorem{assumption}{Assumption}[section]

\newtheorem{theorem}{Theorem}[section]
\newtheorem{lemma}{Lemma}[section]
\newtheorem{definition}{Definition}[section]

\newtheorem{remark}{Remark}[section]

\newtheorem{fact}{Fact}[section]

\makeatletter

\@addtoreset{equation}{section} \makeatother

\usepackage{caption}
\begin{document}

\title{On proximal gradient mapping and its minimization in norm via potential function-based acceleration}

\author{Beier Chen \thanks{
        Department of Mathematics, National University of Defense Technology,
        Changsha, Hunan 410073, China.  Email: \texttt{chenbeier18@nudt.edu.cn}
    }
    \and Hui Zhang \thanks{Department of Mathematics, National University of Defense Technology,
        Changsha, Hunan 410073, China.  Email: \texttt{h.zhang1984@163.com}
    }
}

\date{\today}

\maketitle

\begin{abstract}
    The proximal gradient descent method, well-known for composite optimization, can be completely described by the concept of proximal gradient mapping. In this paper, we highlight our previous two discoveries of proximal gradient mapping--norm monotonicity and refined descent, with which we are able to extend the recently proposed potential function-based framework from gradient descent to proximal gradient descent.
\end{abstract}



%% \linenumbers

%% main text\
\section{Introduction}
First-order methods, which go back to 1847 with the work of Cauchy on the vanilla gradient descent, have recently revived a great deal of research interest due to their low iteration cost as well as low memory storage.
How to establish convergence criteria and determine convergence rates for a given first-order method heavily depends on the choice of optimality measures. The standard optimization literature on smooth convex first-order optimization mainly provides guarantees for  optimality gap (in terms of function value) and distance gap (between the iterate and the minimizer set). However, these optimality measures only have theoretical value but do not fit practical applications because the optimal function value and the minimizer set are usually unknown before applying first-order methods.

Due to the basic fact that minimizing a smooth convex function is equivalent to minimizing the norm of its gradient, a more practical alternative to the optimality gap and distance gap may be the norm of gradient. This fact was initially exploited by Nesterov in the work \cite{nesterov2012}, which argued that using the norm of gradient as an optimality measure is natural and more practical. There are many different potential function-based frameworks covering broad classes of first-order methods for providing optimality gap and distance gap guarantees, but not for the norm of gradient.
This absence motives the authors of \cite{Jelena2022} to introduce a novel potential function-based framework, with which they are able to address the problem of minimizing the norm of the gradient of a smooth convex function. As a natural development, we wonder whether their potential function-based framework can be extended to other types of first-order methods.

In this paper, we go a small step further along this direction by applying their potential function-based framework to composite optimization--minimizing the sum of a smooth convex function and a possibly nonsmooth convex function. To this end, we first revisit the proximal gradient mapping and highlight our previous two discoveries--norm monotonicity and refined descent; both of them may have an independent interest in their own. Then, built on the newly discovered properties and the potential function-based framework of \cite{Jelena2022}, we establish the sublinear convergence for the norm sequence of proximal gradient mapping. Moreover,  we construct a new potential function  to obtain faster convergence.

At the time of writing this paper, a closely related work \cite{li2022proximal}, posted on arXiv very recently, also addressed the problem of minimizing the proximal gradient mapping under the name of proximal subgradient norm minimization. Here, we would like to point out three main differences between this work and ours. First, the potential function-based frameworks are different: they followed the discrete Lyapunov function in \cite{2018Understanding} while we extended that in \cite{Jelena2022}. Second, the accelerated algorithmic schemes are different: they analyzed the faster iterative shrinkage-thresholding algorithm (FISTA) in \cite{Beck2009} while we run two iterative processes for acceleration. At last, the main results are different: they never used the norm  monotonicity of proximal gradient mapping so that their result on proximal subgradient norm minimization for ISTA seems suboptimal. Nevertheless, we believe that these two works have their own merits and complement each other.

The remainder of the paper is organized as follows. In Section 2, we present the basic notation and preliminary knowledge of different function classes, the proximal gradient method, and the potential function-based framework of \cite{Jelena2022}. In Section 3, we revisit the proximal gradient mapping and establish two new properties. In Section 4, we study the problem of minimizing the norm of proximal gradient mapping and show convergence results. Finally, section 5 gives some concluding remarks.

\section{Preliminaries and preliminary results}
In this paper, we restrict our attention to an arbitrary finite dimensional space $\RR^d$ associated with inner product $\langle \cdot, \cdot\rangle$ and norm  $\|\cdot\|:=\sqrt{\langle \cdot, \cdot\rangle}$. For a closed subset $Q\subseteq \RR^d$ and a point $x\in \RR^d$, we define by
$d(x,Q):=\inf_{y\in Q}\|x-y\|$ the distance function from $x$ to $Q$, and define the indicator function of $Q$ by
\begin{eqnarray*}
    \delta_Q(x):=
    \left\{\begin{array}{lll}
        0,       & \textrm{if} ~~x\in Q; \\
        +\infty, & \textrm{otherwise}.
    \end{array} \right.
\end{eqnarray*}

\subsection{Different classes of functions}
In order to introduce the class of smooth convex functions, we first give the definitions of convexity and  smoothness. There are several equivalent definitions of convexity; here we present the first-order definition of convexity in the following form:
\begin{equation}\label{eq:GI}
    (\forall x,y\in\mathbb{R}^n):\quad f(y)\geq f(x)+\left\langle \nabla f(x),y-x\right\rangle.
\end{equation}
The convexity of $f$ essentially says that the function $f$ can be lower bounded by a linear function; In contrast, the smoothness of $f$ actually says that the function $f$ can be upper bounded by a quadratic function, that is
\begin{equation}\label{eq:Lsmooth}
    (\forall x,y\in\mathbb{R}^n):\quad f(y)\leq f(x)+\left\langle \nabla f(x),y-x\right\rangle+\frac{L}{2}\|y-x\|^2,
\end{equation}
where $L>0$ is a constant.
A function is called smooth convex if the inequalities \eqref{eq:GI} and \eqref{eq:Lsmooth} hold at the same time; the class of smooth convex functions is denoted by $\mathcal{F}^{1,1}_L(\mathbb{R}^n)$.  Surprisingly, the convexity inequality \eqref{eq:GI} and the smoothness inequality \eqref{eq:Lsmooth} can be equivalently characterized by a single inequality, that is
\begin{equation}\label{unif1}
    (\forall x,y\in\mathbb{R}^n):\quad f(y)\geq f(x)+\left\langle \nabla f(x),y-x\right\rangle+\frac{1}{L}\|\nabla f(x)-\nabla f(y)\|^2,
\end{equation}
from which the convexity is obviously implied. Interestingly, the inequality \eqref{unif1} is also equivalent to the cocoercive  property of gradient, formulated as
\begin{equation}\label{unif1p}
    (\forall x,y\in\mathbb{R}^n):\quad  \left\langle \nabla f(x)-\nabla f(y), x-y \right\rangle \geq \frac{1}{2L}\|\nabla f(x)-\nabla f(y)\|^2.
\end{equation}

The fact of equivalence between the inequalities above was observed in the book \cite{nesterov2004introductory}. In order to describe a more general fact, we introduce the first-order definition of strong convexity in the form
\begin{equation}
    (\forall x,y\in\mathbb{R}^n):\quad f(y)\geq f(x)+\left\langle \nabla f(x),y-x\right\rangle+\frac{\mu}{2}\|y-x\|^2,
\end{equation}
where $\mu\geq 0$ is a constant, called modulus of strong convexity. In particular, for $\mu=0$ the strong convexity reduces to convexity. In this sense, strong convexity with constant $\mu$ is more general than convexity and hence a wider class of functions, denoted by $\mathcal{S}^{1,1}_{\mu,L}(\mathbb{R}^n)$ and called $L$-smooth and $\mu$-strongly convex, follows. As a matter of fact, we have
$$\mathcal{S}^{1,1}_{\mu=0,L}(\mathbb{R}^n)=\mathcal{F}^{1,1}_L(\mathbb{R}^n).$$
Now, the following statement extends the basic fact that convexity and smoothness is equivalent to \eqref{unif1} or \eqref{unif1p}; for more details please refer to \cite{2021New}.

\begin{fact}\label{fac:ML}
    Let $f:\RR^n\rightarrow \RR$ be a given real-valued function. Then, $f\in\mathcal{S}^{1,1}_{\mu,L}(\mathbb{R}^n)$ if and only if one of the following inequalities holds:
    \begin{equation}\label{unif2}
        (\forall x,y\in\mathbb{R}^n):\quad \left\langle \nabla f(x)-\nabla f(y), x-y \right\rangle \geq \frac{\mu L}{\mu+ L}\|x-y\|^2+\frac{1}{\mu+L}\|\nabla f(x)-\nabla f(y)\|^2,
    \end{equation}
    \begin{equation}\label{unif3}
        \begin{aligned}
            (\forall x,y\in\mathbb{R}^n):\quad  f(x)\geq & f(y)+\left\langle \nabla f(y), x-y\right\rangle+ \frac{1}{2L} \|\nabla f(x)-\nabla f(y)\|^2 \\
                                                         & +\frac{\mu L}{2(L-\mu)}\|x-y-\frac{1}{L}(\nabla f(x)-\nabla f(y))\|^2,
        \end{aligned}
    \end{equation}
    and
    \begin{equation}\label{unif4}
        \mu\|x-y\|\leq \|\nabla f(x)-\nabla f(y)\|\leq L\|x-y\|.
    \end{equation}
\end{fact}
At last, we let $\Gamma_0(\mathbb{R}^n)$ be the class of proper closed and convex functions from $\mathbb{R}^n$ to $(-\infty,+\infty]$. For any $g\in \Gamma_0(\mathbb{R}^n)$, its subdifferential at $x$ is given by
$$\partial g (x):= \{ y \in \RR^n: g(u)\geq g(x)+ \langle y, u-x\rangle,\quad \forall u\in \RR^n \}.$$
The inequality $g(u)\geq g(x)+ \langle y, u-x\rangle$ is called subgradient inequality, each vector in $\partial g(x)$ is called a subgradient of $g$ at $x$.

\subsection{The proximal gradient method}
The proximal gradient method, also called the forward-backward splitting method, is a well-known method for minimizing the sum of a smooth function and a non-smooth function. In the paper, we will be concerned with the following composite optimization
\begin{equation}\label{eg:CO}
    \min_{x\in\mathbb{R}^n}  \varphi(x):=f(x)+g(x),
\end{equation}
where we assume the following.
\begin{assumption} [Composite model assumption] \label{assum1}
    The component functions $f$ and $g$ satisfy that
    \begin{enumerate}
        \item[(A)] $f\in\mathcal{S}^{1,1}_{\mu,L}(\mathbb{R}^n)$, i.e., $f$ is a $L$-smooth and $\mu$-strongly convex function,
        \item[(B)] $g\in \Gamma_0(\mathbb{R}^n)$, i.e., $g$ is a proper closed convex function but it is possibly not smooth, and
        \item[(C)] $X^*$, the set of optimal solutions to \eqref{eg:CO}, is nonempty. The optimal value of the problem is denoted by $\bar{\varphi}$.
    \end{enumerate}
\end{assumption}

Before introducing the concrete iterative scheme of the proximal gradient method, we first give the definition of proximal gradient mapping.
\begin{definition}[PG mapping]\label{def:PGM}
    Suppose that $f$ and $g$ satisfy properties (A) and (B) of Assumption \ref{assum1}. Then the proximal gradient mapping is the operator $\mathcal{G}:\RR^n\times \RR_{+}\rightarrow \RR^n$ defined by
    \begin{equation}\label{eq:PGM}
        \mathcal{G}(x,t):=t^{-1}\left(x-\prox_{tg}(x-t\nabla f(x))\right),
    \end{equation}
    where $\prox_{tg}:\RR^n\rightarrow \RR^n$ is the proximal mapping given by
    $$\prox_{tg}(x):=\arg\min_{y\in\RR^n}\{g(y)+\frac{1}{2t}\|y-x\|^2\}.$$
    In particular, when $g=\delta_Q$, the proximal gradient mapping reduces to the gradient mapping in \cite{nesterov2004introductory}.
\end{definition}
Now, the proximal gradient method, originally given by
$$x^{k+1}=\prox_{t_kg}\left(x^k-t_k\cdot\nabla f(x^k)\right),$$
can be equivalently written into the following form
\begin{equation}\label{eq:PGMM}
    x^{k+1}=x^k-t_k\cdot\mathcal{G}(x^k,t_k).
\end{equation}

\subsection{The potential function-based framework}
The authors of \cite{Jelena2022} introduced a novel potential function-based framework to study the convergence of standard gradient-type methods for making the gradients small in smooth convex optimization. In this part, we first review how their method applies to the standard gradient descent for minimizing a smooth convex function $f\in \mathcal{F}^{1,1}_L(\mathbb{R}^n)$. The key ingredient is that they constructed a potential function of the form
$$\cC_k=\frac{k}{L}\|\nabla f(x^k)\|^2+f(x^k),$$
where the sequence $\{x^k\}_{k\geq 0}$ is generated by the standard gradient descent method, i.e.,
$$x^{k+1}=x^k-\frac{1}{L}\nabla f(x^k), ~~\forall k\geq 0.$$
By invoking the inequalities \eqref{unif1} or \eqref{unif1p}, they can show that the sequence $\{\cC_k\}_{k\geq 0}$ is nonincreasing with $k$ and hence can conclude that $\forall k\geq 0$
$$\|\nabla f(x^k)\|^2\leq \frac{2L(f(x^0)-f(x^*))}{2k+1},$$
where $x^0$ is an arbitrary initial point and $x^*$ is a minimizer of $f$.  In order to design a faster  method than the standard gradient descent, they considered a different potential function of the form
\begin{equation}\label{pfun}
    \cC_k=\sum_{i=0}^{k-1}a_i\|\nabla f(x^i)\|^2+B_k(f(x^k)-f(x^*)),
\end{equation}
where $a_i>0$ ($\forall i\geq0$) the sequence of scalars $B_k>0$ ($\forall k\geq 0$) is strictly increasing,
and the sequence $\{x^k\}_{k\geq 0}$ is generated by the following fast gradient method
\begin{equation}\label{eq:AG}\tag{FGM}
    \left\{\begin{aligned}
        v^k  := & v^{k-1}-\frac{b_{k-1}}{L}\cdot\nabla f(x^{k-1}),                                               \\
        x^k  := & \frac{B_{k-1}}{B_k}\left(x^{k-1}-\frac{1}{L}\cdot\nabla f(x^{k-1})\right) +\frac{b_k}{B_k}v^k,
    \end{aligned}\right.
\end{equation}
with a given arbitrary initial point $x^0$ and $v^0=x^0$. Under some restrictions on the parameters $a_i$ and $B_k$,  invoking again the inequalities \eqref{unif1} and \eqref{unif1p} they showed that
$$\cC_{k+1}-\cC_k\leq  \frac{L}{2}(\|x^*-v^k\|^2-\|x^*-v^{k+1}\|^2), \forall k\geq 0,$$
from which both convergences in function value and in norm of gradient can be obtained. As pointed out, their analysis is the first one that simultaneously leads to both convergence guarantees.

\section{New properties on proximal gradient mapping}
In this section, we first introduce three basic properties of proximal gradient mapping, whose proofs are postponed to Appendix. Then, we highlight two new properties, both of which were discovered in \cite{2019A} by the second author of this paper and posted on arXiv three years ago  but they have not yet been submitted for publication.
\subsection{Basic lemmas}
The first lemma is an equivalent characterization of the proximal mapping.
\begin{lemma}\label{lem:zero}
    Let $g\in\Gamma_0(\mathbb{R}^n)$ and $t>0$. Thus
    $z=\prox_{tg}(y)$
    if and only if $y\in(I+t\cdot\partial g)(z)$.
\end{lemma}

The second lemma provides the relationship between the norm of proximal gradient mapping and the smallest norm of  subgradient.
\begin{lemma}\label{lem:UB}
    Suppose that $f$ and $g$ satisfy properties (A) and (B) of Assumption \ref{assum1}. For any $x\in\mathbb{R}^n$ and $t>0$, we have
    \begin{equation}\label{eq:UB}
        \|\mathcal{G}(x,t)\|\leq d(0,\partial\varphi(x)).
    \end{equation}
\end{lemma}

The last lemma is a slight modification of the classic descent lemma, originally discovered by Beck and Teboulle in \cite{Beck2009}. It also extends Corollary 2.3.2 in \cite{nesterov2004introductory} from gradient mapping to proximal gradient mapping. When $\mu=0$, it reduces to the  pivotal inequality in the recent work \cite{li2022proximal}.
\begin{lemma}\label{lem:OVG}
    Suppose that $f$ and $g$ satisfy properties (A) and (B) of Assumption \ref{assum1}. Then, we have
    \begin{equation}\label{eq:OVG}
        \varphi(x)-\varphi(y-t\mathcal{G}(y,t))\geq t(1-\frac{L}{2}t)\| \mathcal{G}(y,t)\|^2+\left\langle\mathcal{G}(y,t),x-y\right\rangle+\frac{\mu}{2}\|x-y\|^2.
    \end{equation}
    In particular, the inequality above with $t=\frac{1}{L}$ and $\mu=0$ in (\ref{eq:OVG}) yields
    \begin{equation}\label{eq:OVGL}
        \varphi(x)-\varphi(y-\frac{1}{L}\mathcal{G}(y,\frac{1}{L}))\geq \frac{1}{2L}\|\mathcal{G}(y,\frac{1}{L})\|^2+\left\langle\mathcal{G}(y,\frac{1}{L}), x-y\right\rangle.
    \end{equation}
\end{lemma}

\subsection{New and refined results}
For simplicity, we define the updated iterate point by using the superscript "+" as follows:
$$\quad x^+:=\prox_{tg}(x-t\nabla f(x))=x-t\cdot\mathcal{G}_t(x),$$
where the step size $t>0$ is clear from the context.
Using this notation and Lemma \ref{lem:zero}, we immediately  have
$$x-t\nabla f(x)\in x^++ t\partial g(x^+).$$
Thus, there must exist a subgradient  $s^+\in\partial g(x^+)$ such that
\begin{equation}\label{eq:splus}
    x^+=x-t(\nabla f(x)+s^+).
\end{equation}
Now, we are ready to present the first new property of proximal gradient mapping.
\begin{theorem}[Norm monotonicity]\label{lem:PGN}
    Suppose that $f$ and $g$ satisfy properties (A) and (B) of Assumption \ref{assum1}.
    Denote $\rho(t):=\max\{ |1-Lt|, |1-\mu t|\}$. Then,  we have
    \begin{equation}\label{eq:PGN}
        \|\mathcal{G}(x^+,t)\|\leq d(0, \partial \varphi(x^+))\leq\rho(t)\|\mathcal{G}(x,t)\|\leq \rho(t) d(0, \partial \varphi(x)).
    \end{equation}
    In particular, for $f\in\mathcal{F}^{1,1}_{L}(\mathbb{R}^n)$, $g\in\Gamma_0(\mathbb{R}^n)$, and $0< t\leq \frac{2}{L}$, it holds that
    $$\|\mathcal{G}(x^+,t)\|\leq d(0, \partial \varphi(x^+))\leq \|\mathcal{G}(x,t)\| \leq d(0, \partial \varphi(x)).$$
\end{theorem}

\begin{proof}
    The inequalities $\|\mathcal{G}(x^+,t)\|\leq d(0, \partial \varphi(x^+))$ and $\rho(t)\|\mathcal{G}(x,t)\|\leq \rho(t) d(0, \partial \varphi(x))$ directly follow from Lemma \ref{lem:UB}. To show the relationship (\ref{eq:PGN}), it suffices to show that
    \begin{equation}\label{eq:Mineq}
        d(0, \partial \varphi(x^+))\leq\rho(t)\|\mathcal{G}(x,t)\|.
    \end{equation}
    Since $s^+ \in \partial g(x^+)$, we have  $d(0,\partial \varphi(x^+)) \le \| \nabla f(x^+) + s^+\|$.  Therefore, if we can show that
    \begin{equation}\label{eq:Mineq2}
        \|\nabla f(x^+)+s^+\|^2\leq \rho^2(t)\|\mathcal{G}(x,t)\|^2,
    \end{equation}
    then the desired inequality \eqref{eq:Mineq} follows immediately. Using he expression $x^+=x-t(\nabla f(x)+s^+)$ in \eqref{eq:splus}, we derive that
    \begin{align*}
             & \|\nabla f(x^+)+s^+\|^2                                                                                                                                            \\
        =    & \|\nabla f(x)+s^+ +\nabla f(x^+)-\nabla f(x)\|^2                                                                                                                   \\
        =    & \|\nabla f(x)+s^+\|^2+2\left\langle \nabla f(x)+s^+, \nabla f(x^+)-\nabla f(x) \right\rangle +  \|\nabla f(x^+)-\nabla f(x)\|^2                                    \\
        =    & \frac{1}{t^2}\|x^+-x\|^2-\frac{2}{t} \left\langle x^+-x, \nabla f(x^+)-\nabla f(x) \right\rangle +  \|\nabla f(x^+)-\nabla f(x)\|^2                                \\
        \leq & \frac{1}{t^2}\|x^+-x\|^2-\frac{2}{t}\left(\frac{\mu L}{\mu+ L}\|x^+-x\|^2+\frac{1}{\mu+L}\|\nabla f(x^+)-\nabla f(x)\|^2\right) +  \|\nabla f(x^+)-\nabla f(x)\|^2 \\
        =    & \frac{1}{t^2}\left[(1-\frac{2t\mu L}{\mu+ L})\|x^+-x\|^2 + t(t- \frac{2}{\mu+L})\|\nabla f(x^+)-\nabla f(x)\|^2 \right],
    \end{align*}
    where the inequality follows from \eqref{unif2} in Fact \ref{fac:ML}. In order to bound $\|\nabla f(x^+)-\nabla f(x)\|^2$ in terms of $\|x^+-x\|^2$, we use \eqref{unif4} in Fact \ref{fac:ML} to get
    $$\mu^2 \|x^+-x\|^2\leq \|\nabla f(x^+)-\nabla f(x)\|^2\leq L^2\|x^+-x\|^2.$$
    If $t- \frac{2}{\mu+L}\geq 0$, then we have
    $$(t- \frac{2}{\mu+L})\|\nabla f(x^+)-\nabla f(x)\|^2\leq L^2(t- \frac{2}{\mu+L})\|x^+-x\|^2.$$
    If $t- \frac{2}{\mu+L}<0$, then we have
    $$(t- \frac{2}{\mu+L})\|\nabla f(x^+)-\nabla f(x)\|^2\leq \mu^2(t- \frac{2}{\mu+L})\|x^+-x\|^2.$$
    In both cases, we always have that
    $$(t- \frac{2}{\mu+L})\|\nabla f(x^+)-\nabla f(x)\|^2\leq \max\left\{L^2(t- \frac{2}{\mu+L}), \mu^2(t- \frac{2}{\mu+L})\right\}\|x^+-x \|^2.$$
    Therefore, we can continue to derive that
    \begin{align*}
             & \|\nabla f(x^+)+s^+\|^2                                                                                                                        \\
        \leq & \frac{1}{t^2}\left[(1-\frac{2t\mu L}{\mu+ L})\|x^+-x\|^2 + t\max\left\{L^2(t- \frac{2}{\mu+L}), \mu^2(t- \frac{2}{\mu+L})\right\}\|x^+-x \|^2 \right]     \\
        =    & \frac{1}{t^2}\max\left\{1-\frac{2t\mu L}{\mu+ L} + tL^2(t- \frac{2}{\mu+L}), 1-\frac{2t\mu L}{\mu+ L}  +  t\mu^2(t- \frac{2}{\mu+L}) \right\}\|x^+-x \|^2 \\
        =    & \frac{1}{t^2} \max\{ (1-Lt)^2, (1-\mu t)^2\}\|x^+-x \|^2                                                                                       \\
        =    & \rho^2(t)\|\mathcal{G}(x,t)\|^2,
    \end{align*}
    from which the inequality \eqref{eq:Mineq2} follows. This completes the proof.
\end{proof}
\begin{remark}
    Here, the factor $\rho(t)$ is optimal; otherwise, it will contradict the following exact worst-case convergence rate, which was recently established in \cite{Taylor2018}:
    $$\|\nabla f(x^+)+s^+\|\leq \rho(t)\|\nabla f(x)+s\|, ~~\forall s\in \partial g(x).$$
    In fact, the inequality above is equivalent to
    $$\|\nabla f(x^+)+s^+\|\leq \rho(t)d(0,\partial \varphi(x));$$
    whilst in our proof, we have shown $ \|\nabla f(x^+)+s^+\|\leq \rho(t)\|\mathcal{G}(x,t)\|$ in \eqref{eq:Mineq2} which is a tighter estimation and hence it is impossible to improve.
\end{remark}

Below, we state the second new property of proximal gradient mapping.
\begin{theorem}[Refined descent]\label{lem:SDP}
    Suppose that $f$ and $g$ satisfy properties (A) and (B) of Assumption \ref{assum1}. Then, we have
    \begin{equation}\label{eq:SDP}
        \varphi(x)\geq \varphi(x^+) +\frac{t}{2}\|\mathcal{G}(x,t)\|^2 + \frac{t}{2(1-\mu t)}\|\mathcal{G}(x^+,t)\|^2, 0< t\leq \frac{1}{L}.
    \end{equation}
    In particular,
    \begin{itemize}
        \item for $f \in \mathcal{F}^{1,1}_{L}(\RR^n)$, $g \in \Gamma_0(\RR^n)$, it holds that
              \begin{equation}\label{eq:DP1}
                  \varphi(x)\geq \varphi(x^+) +\frac{t}{2}\|\mathcal{G}(x,t)\|^2 + \frac{t}{2}\|\mathcal{G}(x^+,t)\|^2, 0< t\leq \frac{1}{L}.
              \end{equation}
        \item  for $f \in \mathcal{F}^{1,1}_{L}(\RR^n)$, $g \equiv 0$, it holds that
              \begin{equation}\label{eq:DP2}
                  f(x) \geq f(x^+) +\frac{t}{2}\|\nabla f(x)\|^2 + \frac{t}{2}\|\nabla f(x^+)\|^2, 0< t \leq \frac{1}{L}.
              \end{equation}
    \end{itemize}
\end{theorem}
\begin{proof}
    Note that $0<t \leq L^{-1}$ implies $t^{-1}\geq L$ which further implies that the $L$-smooth function must also be  $t^{-1}$-smooth; hence, we can conclude that $$\cS^{1,1}_{\mu,L}(\RR^n) \subset\cS^{1,1}_{\mu, t^{-1}}(\RR^n).$$
    We now use \eqref{unif3} in Fact \ref{fac:ML} with $L=t^{-1}$ and $y=x^+$  to get
    $$f(x)\geq f(x^+)+\left\langle \nabla f(x^+), x-x^+\right\rangle+ \frac{t}{2} \|\nabla f(x)-\nabla f(x^+)\|^2+\frac{\mu }{2(1-\mu t)}\|x-x^+-t(\nabla f(x)-\nabla f(x^+))\|^2.$$
    The subgradient inequality of $g$ gives $g(x)\geq g(x^+)+\left\langle s^+, x-x^+\right\rangle$ since $s^+ \in \partial g(x^+)$. Adding up these two inequalities, we derive that
    \begin{align*}
        \varphi(x)\geq & \varphi(x^+)+\left\langle \nabla f(x^+)+s^+, x-x^+\right\rangle+ \frac{t}{2} \|\nabla f(x)-\nabla f(x^+)\|^2                 \\
                       & +\frac{\mu }{2(1-\mu t)}\|x-x^+-t(\nabla f(x)-\nabla f(x^+))\|^2                                                             \\
        =              & \varphi(x^+)+\left\langle \nabla f(x)+s^+, x-x^+\right\rangle - \left\langle \nabla f(x^+)- \nabla f(x), x^+ -x\right\rangle \\
                       & + \frac{t}{2} \|\nabla f(x)-\nabla f(x^+)\|^2 +\frac{\mu }{2(1-\mu t)}\|x-x^+-t(\nabla f(x)-\nabla f(x^+))\|^2
    \end{align*}
    Using the expression $x^+=x-t(\nabla f(x)+s^+)$ in \eqref{eq:splus}, we can further derive that
    \begin{align*}
        \varphi(x) \geq & \varphi(x^+)+\frac{1}{t}\|x-x^+\|^2 - \left\langle \nabla f(x^+)- \nabla f(x), x^+ -x\right\rangle \\
                        & + \frac{t}{2} \|\nabla f(x)-\nabla f(x^+)\|^2 +\frac{\mu t^2}{2(1-\mu t)}\|s^+ + \nabla f(x^+)\|^2 \\
        =               & \varphi(x^+) + \frac{1}{2t}\|t(\nabla f(x^+)- \nabla f(x))-x^+ +x\|^2                              \\
                        & +\frac{1}{2t}\|x-x^+\|^2 +\frac{\mu t^2}{2(1-\mu t)}\|s^+ + \nabla f(x^+)\|^2                      \\
        =               & \varphi(x^+)+\frac{1}{2t}\|x-x^+\|^2  +\frac{t}{2(1-\mu t)}\|s^+ + \nabla f(x^+)\|^2.
    \end{align*}
    Note that $x-x^+=t \mathcal{G}(x,t)$ and use the fact that
    $$\|s^+ + \nabla f(x^+)\|\geq d(0, \partial \varphi(x^+))\geq \|\mathcal{G}(x^+,t)\|.$$
    We finally obtain
    $$\varphi(x)\geq \varphi(x^+) +\frac{t}{2}\|\mathcal{G}(x,t)\|^2 + \frac{t}{2(1-\mu t)}\|\mathcal{G}(x^+,t)\|^2.$$
    This completes the proof.
\end{proof}

\begin{remark} We make a few remarks:
    \begin{itemize}
        \item    In \cite{nesterov2004introductory}, for $\varphi = f+ g$ with $f \in \cS^{1,1}_{\mu,L}(\RR^n)$ and $g$ being the indicator function of a set $Q$, the descent lemma of the projected gradient method can be stated as
              \begin{equation}\label{eq:compare1}
                  \varphi(x) \geq \varphi(x^+) +\frac{t}{2}\|g_Q(x,t)\|^2, 0 < t \le \frac{1}{L}.
              \end{equation}
              where $g_Q(x, t) := t^{-1}(x-x^+)$ is the gradient mapping of $f$ on $Q$. In \cite{Beck2009}, for $\varphi = f+g$ with $f  \in \mathcal{F}^{1,1}_{L}(\RR^n)$ and $g \in \Gamma_0(\RR^n)$, the corresponding descent lemma of the proximal gradient method is
              \begin{equation}\label{eq:compare2}
                  \varphi(x) \geq \varphi(x^+) +\frac{L}{2}\|x^+ - x\|^2.
              \end{equation}
              It is not hard to see that our result improves these existing descent lemmas.

        \item   At the time of this paper was under preparation, we noticed that the special case \eqref{eq:DP1} was implicitly rediscovered by combining Lemma 9 and Lemma 11 in \cite{2022Teboulle}.
    \end{itemize}

\end{remark}

\section{Small norm of proximal gradient mapping}
In this section, we aim to extend the potential function-based framework previously reviewed from gradient descent to proximal gradient descent and its acceleration.

\subsection{Proximal gradient descent}
The following result is a direct extension of Lemma 2.1 in \cite{Jelena2022}. However, its proof relies on the new properties of proximal gradient mapping in the last section.

\begin{theorem}\label{MNG}
    Suppose that Assumption \ref{assum1} holds. Let $x^0$ be an arbitrary initial point and assume that $x^{k+1}=x^k-t_k\mathcal{G}(x^k,t_k)$ with constant step sizes $t_k\equiv\frac{\eta}{L}$ for some $0<\eta\leq 1$. Then
    $$\mathcal{C}_k:=\frac{\eta}{L}\cdot k\|\mathcal{G}(x^k,\frac{\eta}{L})\|^2+\varphi(x^k)-\bar{\varphi}$$
    is nonincreasing with $k$, and the norm sequence of proximal gradient mappings converges sublinearly in the sense that $\forall k\geq0$
    $$\|\mathcal{G}(x^k,\frac{\eta}{L})\|\leq \frac{L(\varphi(x^0)-\bar{\varphi})}{\eta k}.$$
\end{theorem}
\begin{proof}
    We first show that $\forall k\geq 0$,
    $$\mathcal{C}_{k+1}\leq \mathcal{C}_k.$$
    Using the definition of $\mathcal{C}_k$, we have that
    $$\mathcal{C}_{k+1}-\mathcal{C}_k=\frac{\eta}{L}(k+1)\|\mathcal{G}(x^{k+1},\frac{\eta}{L})\|^2-\frac{\eta k}{L} \|\mathcal{G}(x^k,\frac{\eta}{L})\|^2+\varphi(x^{k+1})-\varphi(x^k).$$
    Using Theorem \ref{lem:SDP} yields
    $$\varphi(x^k)-\varphi(x^{k+1})\geq \frac{\eta}{2L}\|\mathcal{G}(x^{k},\frac{\eta}{L})\|^2+\frac{\eta}{2L}\|\mathcal{G}(x^{k+1},\frac{\eta}{L})\|^2.$$
    Thus,
    $$\mathcal{C}_{k+1}-\mathcal{C}_k\leq \frac{\eta}{L}\left(k+\frac{1}{2}\right)(\| \mathcal{G}(x^{k+1},\frac{\eta}{L})\|^2-\| \mathcal{G}(x^k,\frac{\eta}{L})\|^2).$$
    In addition, using Theorem \ref{lem:PGN} yields
    $$\|\mathcal{G}(x^{k+1},\frac{\eta}{L})\|\leq \|\mathcal{G}(x^k,\frac{\eta}{L})\|,$$
    which leads to the monotonically decreasing $\mathcal{C}_{k+1}\leq \mathcal{C}_k$ and the result
    $$\varphi(x^k)-\bar{\varphi}+\frac{\eta k}{L}\cdot\|\mathcal{G}(x^k,\frac{\eta}{L})\|^2\leq \cdots\leq \mathcal{C}_0=\varphi(x^0)-\bar{\varphi}.$$
    Equivalently,
    $$\frac{\eta k}{L}\|\mathcal{G}(x^k,\frac{\eta}{L})\|^2\leq \varphi(x^0)-\varphi(x^k)\leq \varphi(x^0)-\bar{\varphi},$$
    from which the conclusion follows.
\end{proof}

\subsection{Accelerated norm minimization}
We start with the following iterative scheme which is obtained by replacing the gradient in the fast gradient method \eqref{eq:AG} by the proximal gradient mapping and introducing a new sequence $\{y^k\}$.  For any $k\geq1$,
\begin{equation}\label{eq:APG}\tag{APG}
    \left\{\begin{aligned}
         & y^{k-1}  := x^{k-1}-\frac{1}{L}\cdot\mathcal{G}(x^{k-1},\frac{1}{L}),       \\
         & v^k      := v^{k-1}-\frac{b_{k-1}}{L}\cdot\mathcal{G}(x^{k-1},\frac{1}{L}), \\
         & x^k     := \frac{B_{k-1}}{B_k}y^{k-1}+\frac{b_k}{B_k}v^k,
    \end{aligned}\right.
\end{equation}
where the sequence of scalars $B_k>0$ will be determined later and the sequence of scalars $b_k$ is defined by $b_0=B_0$ and $b_k=B_k-B_{k-1}$ for $k\geq 1$.
For simplicity, we let $\mathcal{G}(x^k)\equiv \mathcal{G}(x^k,\frac{1}{L})$ be the proximal gradient mapping when the step size $t$ equals to $\frac{1}{L}$. Our forthcoming analysis mainly relies on the following potential function: for any $k\geq 0$
\begin{equation}\label{eq:PFL}
    \mathcal{C}_k:=\sum\limits^k_{i=0}a_i\|\mathcal{G}(x^i)\|^2+B_k(\varphi(y^k)-\bar{\varphi}),
\end{equation}
which is inspired by the potential function \eqref{pfun}. However, when zooming into the expression more carefully, the reader can find that it is not obtained by simply replacing the gradient in \eqref{pfun} by the proximal gradient mapping. Actually, we use the function value at $y^k$ rather than at $x^k$ and the sum is from $i=0$ to $k$ rather than to $k-1$. These modifications are pivotal to deduce our desired conclusions.

\begin{lemma}\label{lem:DIP}
    Suppose that Assumption \ref{assum1} holds. Let $x^0$ be an arbitrary initial point with $v^0=x^0$ and assume that
    the sequences of $\{x^k\}$, $\{y^k\}$, and $\{v^k\}$ are generated by the algorithm \eqref{eq:APG}. If the nonnegative scalars $a_k, b_k, B_k$ satisfy that $\forall k\geq1$,
    $$a_k\leq \frac{B_k-b_k^2}{2L},$$
    then we have
    $$\mathcal{C}_k-\mathcal{C}_{k-1}\leq \frac{L}{2}(\| x^*-v^k\|^2-\| x^* -v^{k+1}\|^2), \forall k\geq1,$$
    where $x^* \in X^*$.
\end{lemma}

\begin{proof}
    Using the definition of $\mathcal{C}_k$ in (\ref{eq:PFL}), we have that for any $k\geq 1$,
    \begin{equation}\label{eq:gap}
        \mathcal{C}_k-\mathcal{C}_{k-1}\leq a_k\| \mathcal{G}(x^k)\|^2+B_k\varphi(y^k)-B_{k-1}\varphi(y^{k-1})-b_k\bar{\varphi}.
    \end{equation}
    Now, we use (\ref{eq:OVGL}) in Lemma \ref{eq:OVGL} to bound the unknown optimal function value $\bar{\varphi}$. Actually, the inequality (\ref{eq:OVGL}) with $x=x^* $ and $y=x^k$ gives us
    \begin{equation}\label{eq:lowerb}
        \bar{\varphi}=\varphi(x^*)\geq \varphi(y^k)+\frac{1}{2L}\|\mathcal{G}(x^k)\|^2+\left\langle \mathcal{G},x^*-x^k\right\rangle.
    \end{equation}
    Using (\ref{eq:OVGL}) again with $x=y^{k-1}$ and $y=x^k$ leads to
    \begin{equation}\label{eq:ygap}
        \varphi(y^{k-1})-\varphi(y^k)\geq \frac{1}{2L}\| \mathcal{G}(x^k)\|^2+\left\langle \mathcal{G}(x^k),y^{k-1}-x^k\right\rangle.
    \end{equation}
    Combining the three inequalities above, we derive that
\begin{equation}\label{eq:upperb}
    \begin{aligned}
        \mathcal{C}_k -\mathcal{C}_{k-1} \leq & a_k\|\mathcal{G}(x^k)\|^2+B_k\varphi(y^k)-B_{k-1}\varphi(y^{k-1})-b_k\varphi(y^k)-\frac{b_k}{2L}\|\mathcal{G}(x^k)\|^2                                               \\
                                              & -b_k\left\langle \mathcal{G}(x^k),x^*-x^k\right\rangle                                                                                                           \\
        =                                     & a_k\| \mathcal{G}(x^k)\|^2+B_{k-1}(\varphi(y^k)-\varphi(y^{k-1}))-\frac{b_k}{2L}\|\mathcal{G}(x^k)\|^2                                                               \\
                                              & -b_k\left\langle \mathcal{G}(x^k),x^*-x^k\right\rangle                                                                                                           \\
        \leq                                  & a_k\|\mathcal{G}(x^k)\|^2-\frac{B_{k-1}}{2L}\|\mathcal{G}(x^k)\|^2-B_{k-1}\left\langle\mathcal{G}(x^k),y^{k-1}-x^k\right\rangle-\frac{b_k}{2L}\|\mathcal{G}(x^k)\|^2 \\
                                              & +b_k\left\langle\mathcal{G}(x^k),x^k-x^*\right\rangle                                                                                                            \\
        =                                     & \left(a_k-\frac{B_k}{2L}\right)\|\mathcal{G}(x^k)\|^2+B_{k-1}\left\langle \mathcal{G}(x^k),x^k-x^{k-1}+\frac{1}{L}\mathcal{G}(x^{k-1})\right\rangle                  \\
                                              & +b_k\left\langle \mathcal{G}(x^k),x^k-x^*\right\rangle.
    \end{aligned}
    \end{equation}
    In order to get an acceptable upper bound of $\mathcal{C}_k -\mathcal{C}_{k-1}$, we need to estimate the inner product term $\left\langle\mathcal{G}(x^k),x^k-x^*\right\rangle$.
    This can be done by going through the following arguments which are standard in mirror-descent-type analysis. First, note that
    \begin{equation*}
        \begin{aligned}
            v^{k+1} & =\arg\min\limits_u\left\{b_k\left\langle \mathcal{G}(x^k),u-v^k\right\rangle+\frac{L}{2}\| u-v^k\|^2\right\} \\
                    & =v^k-\frac{b_k}{L}\mathcal{G}(x^k).
        \end{aligned}
    \end{equation*}
    Then, we can deduce that
    \begin{equation}\label{eqmidd}
        \begin{aligned}
            b_k\left\langle\mathcal{G}(x^k),x^k -x^*\right\rangle = & b_k\left\langle \mathcal{G}(x^k),x^k-v^{k+1}\right\rangle+L\left\langle v^k-v^{k+1},v^{k+1}-x^*\right\rangle          \\
            =                                                           & b_k\left\langle \mathcal{G}(x^k),x^k-v^k\right\rangle+\frac{b_k^2}{L}\|\mathcal{G}(x^k)\|^2+\frac{L}{2}\| x^*-v^k\|^2 \\
                                                                        & -\frac{L}{2}\| x^*-v^{k+1}\|^2-\frac{L}{2}\| v^{k+1}-v^k\|^2                                                          \\
            =                                                           & b_k\left\langle \mathcal{G}(x^k),x^k-v^k\right\rangle+\frac{b_k^2}{2L}\|\mathcal{G}(x^k)\|^2                              \\
                                                                        & +\frac{L}{2}\| x^*-v^k\|^2  -\frac{L}{2}\| x^*-v^{k+1}\|^2,
        \end{aligned}
    \end{equation}
    where the relationship $v^{k+1}:= v^{k}-\frac{b_{k}}{L}\cdot\mathcal{G}(x^{k},\frac{1}{L})$ have been repeatedly used.
    Now, combining \eqref{eqmidd}  and (\ref{eq:upperb}), we can get
    \begin{equation}\label{mormid}
        \begin{aligned}
            \mathcal{C}_k-\mathcal{C}_{k-1}  \leq & \left(a_k-\frac{B_k-b_k^2}{2L}\right)\|\mathcal{G}(x^k)\|^2 +\frac{L}{2}\| x^*-v^k\|^2-\frac{L}{2}\| x^*-v^{k-1}\|^2 \\
                                                  & +\left\langle \mathcal{G}(x^k),B_kx^k-B_{k-1}(x^{k-1}-\frac{1}{L}\mathcal{G}(x^{k-1}))-b_kv^k\right\rangle.
        \end{aligned}
    \end{equation}
    Note that
    $$B_kx^k-B_{k-1}\left(x^{k-1}-\frac{1}{L}\mathcal{G}(x^{k-1})\right)-b_kv^k=B_kx^k-B_{k-1}y^{k-1}-b_kv^k=0.$$
    The inner product term \eqref{mormid} disappears and hence using the condition $a_k\leq \frac{B_k-b_k^2}{2L}$ we  finally obtain
    $$\mathcal{C}_k-\mathcal{C}_{k-1} \leq \frac{L}{2}\| x^*-v^k\|^2-\frac{L}{2}\| x^*-v^{k+1}\|^2.$$
    This completes the proof.
\end{proof}

Now, we are ready to present the accelerated convergence of proximal gradient mapping.
\begin{theorem}\label{th:CAPG}
    Suppose that the assumption in \cref{lem:DIP} holds. Denote
    $$\tilde{\mathcal{C}}:=a_0\|\mathcal{G}(x^0)\|^2+b_0(\varphi(y^0)-\bar{\varphi})+\frac{L}{2}\| x^*-v^0\|^2.$$
    Then, we have
    \begin{equation}\label{objvalue}
        \varphi(y^k)-\bar{\varphi}\leq \frac{\tilde{\mathcal{C}}}{B_k},       k\geq 1,
    \end{equation}
    \begin{equation}\label{gradnorm}
        \sum\limits^k_{i=0}a_i\| \mathcal{G}(x^i)\|^2\leq \tilde{\mathcal{C}},  k\geq 1.
    \end{equation}
    In particular, if $b_k=\frac{1}{4}(k+1)$, $B_k=\frac{1}{8}(k+1)(k+2)$, $a_k=\frac{1}{32L}(k+1)^2$ for $k\geq 1$, then
    \begin{equation}\label{eq:1}
        \varphi(y^k)-\bar{\varphi}  \leq \frac{8\tilde{\mathcal{C}}}{(k+1)(k+2)},
    \end{equation}
    and
    \begin{equation}\label{eq:2}
        \min\limits_{0\leq i\leq k}\|\mathcal{G}(x^i)\|^2\leq \frac{192L\tilde{\mathcal{C}}}{(k+1)(k+2)(k+3)}.
    \end{equation}
\end{theorem}

\begin{proof}
    Using Lemma \ref{lem:DIP} and the definition $\cC_k$, we have
    \begin{equation}
        \begin{aligned}
            \mathcal{C}_k & \leq \mathcal{C}_0+\frac{L}{2}\| x^* -v^0\|^2-\frac{L}{2}\| x^* -v^{k+1}\|^2          \\
                          & \leq a_0\|\mathcal{G}(x^0)\|^2+B_0(\varphi(y^0)-\bar{\varphi})+\frac{L}{2}\| x^* -v^0\|^2 \\
                          & =\tilde{\mathcal{C}}.
        \end{aligned}
    \end{equation}
    Note that each term in $\cC_k$ is nonnegative. Thus, for any $k\geq 1$ we can get
    $$B_k(\varphi(y^k)-\bar{\varphi})\leq \mathcal{C}_k\leq\tilde{\mathcal{C}}$$
    and
    $$\sum\limits^k_{i=0}a_i\|\mathcal{G}(x^i)\|^2\leq \mathcal{C}_k\leq \tilde{\mathcal{C}},$$
    from which the first part follows.

    As for the second part, we first show that the condition  $a_k\leq \frac{B_k-b_k^2}{2L}$ can be verified by the current setting $b_k=\frac{1}{4}(k+1)$, $B_k=\frac{1}{8}(k+1)(k+2)$, and $a_k=\frac{1}{32L}(k+1)^2$. In fact,
    \begin{equation*}
        \begin{aligned}
            \frac{B_k-b_k^2}{2L} & =\frac{1}{2L}\left[\frac{1}{8}(k+1)^2+\frac{1}{8}(k+1)-\frac{1}{16}(k+1)^2\right] \\
                                 & \geq \frac{1}{2L}\cdot\frac{1}{16}(k+1)^2                                         \\
                                 & =a_k.
        \end{aligned}
    \end{equation*}
    Now, summing $a_i$ from $i=0$ to $i=k$, we obtain
    $$\sum\limits^k_{i=0}a_i=\sum^k_{i=0}\frac{1}{32L}(i+1)^2=\frac{(k+1)(k+2)(2k+3)}{192L}.$$
    Therefore, combining with \eqref{gradnorm} in the first part, we finally get
    \begin{equation*}
        \min\limits_{0\leq i\leq k}\| \mathcal{G}(x^i)\|^2  \leq \frac{\sum\limits^k_{i=0}a_i\|\mathcal{G}(x^i)\|^2}{\sum\limits^k_{i=0}a_i}  \leq \frac{192L\cdot\tilde{\mathcal{C}}}{(k+1)(k+2)(2k+3)},
    \end{equation*}
    which completes the proof.
\end{proof}
\section{Concluding remarks}
In this paper, we successfully extended the potential function-based framework in \cite{Jelena2022} from gradient descent to proximal gradient descent, with the help of two newly discovered properties on the proximal gradient mapping. However, the modulus of strong convexity has not yet been exploited in the current potential function-based framework to provide linear convergence guarantees for the norm of gradient or proximal gradient mapping; we would like to leave it as future work.

\section*{Acknowledgements}
This work is supported by the National Science Foundation of China (Nos.11971480).

\section*{Appendix: The missing proofs}

%\subsection{The missing proofs in convergence analysis}
\noindent {\bf The proof of Lemma \ref{lem:zero}}:
Using the definition of the proximal mapping yields
$$z=\prox_{tg}(y)=\arg\min\limits_x\{tg(x)+\frac{1}{2}\| x-y\|^2\}.$$
Based on the first-order optimality condition, we have
$$0\in t\cdot\partial g(z)+z-y.$$
Hence, the relationship $y\in(I+t\cdot\partial g)(z)$ follows. This completes the proof.

\bigskip

\noindent {\bf The proof of Lemma \ref{lem:UB}}:
Take a subgradient $s\in\partial\varphi(x)=\partial g(x)+\nabla f(x)$; then, it holds that
$$x-t\nabla f(x)+ts\in(I+t\partial g)(x).$$
Hence, from Lemma \ref{lem:zero} we have
$$x=\prox_{tg}(x-t\nabla f(x)+ts).$$
Using the nonexpansive property of proximal mapping, for any $s\in\partial \varphi(x)$ we have
\begin{equation*}
    \begin{aligned}
        t\|\mathcal{G}(x,t)\| & =\| x-\prox_{tg}(x-t\nabla f(x))\|                             \\
                              & =\| \prox_{tg}(x-t\nabla f(x)+ts)-\prox_{tg}(x-t\nabla f(x))\| \\
                              & \leq t\| s\|,~~\forall s\in\partial\varphi(x),
    \end{aligned}
\end{equation*}
from which the upper bound (\ref{eq:UB}) follows. This completes the proof.

\bigskip

\noindent {\bf The proof of Lemma \ref{lem:OVG}}:
First of all, we define the following auxiliary function
$$h(x,y):=g(x)+f(y)+\left\langle\nabla f(y),x-y\right\rangle+\frac{1}{2t}\| x-y\|^2.$$
Denote
\begin{equation}\label{eq:upy}\tag{A.1}
    y^+:=\arg\min\limits_xh(x,y).
\end{equation}
Then, one can verify that
$$y^+=y-t\mathcal{G}(y,t).$$
Applying the $L$-smoothness of $f$ in (\ref{eq:Lsmooth}), we obtain
\begin{equation*}
    \begin{aligned}
        \varphi(x) & =f(x)+g(x)\leq h(x,y)+(\frac{L}{2}-\frac{1}{2t})\| x-y\|^2.
    \end{aligned}
\end{equation*}
Plugging $x=y^+$ in the above equation, we get
$$\varphi(y^+)\leq h(y^+,y)+(\frac{L}{2}-\frac{1}{2t})\| y^+-y\|^2,$$
or equivalently,
\begin{equation}\label{eq:vg}\tag{A.2}
    \varphi(x)-\varphi(y^+)\geq \varphi(x)-h(y^+,y)-(\frac{L}{2}-\frac{1}{2t})\| y^+-y\|^2.
\end{equation}
Due to the optimality condition of (\ref{eq:upy}), there must exist a subgradient $g_s\in\partial g(y^+)$ such that
\begin{equation}\label{eq:zero}\tag{A.3}
    0=g_s+\nabla f(y)+\frac{1}{t}(y^+-y).
\end{equation}
Invoking the subgradient inequality for $g$ and the $\mu$-strong convexity for $f$, we have
\begin{equation*}
    \begin{aligned}
        f(x) & \geq f(y)+\left\langle\nabla f(y),x-y\right\rangle+\frac{\mu}{2}\|x-y\|^2, \\
        g(x) & \geq g(y^+)+\left\langle g_s,x-y^+\right\rangle.
    \end{aligned}
\end{equation*}
Adding these two inequalities together, we get
\begin{equation}\label{eq:lb}\tag{A.4}
    \varphi(x)\geq f(y)+g(y^+)+\left\langle\nabla f(y),x-y\right\rangle+\left\langle g_s,x-y^+\right\rangle+\frac{\mu}{2}\|x-y\|^2.
\end{equation}
On the other hand,
\begin{equation*}
    h(y^+,y)=g(y^+)+f(y)+\left\langle y^+-y,\nabla f(y)\right\rangle+\frac{1}{2t}\| y^+-y\|^2.
\end{equation*}
Combining the preceding equation with (\ref{eq:vg}) and (\ref{eq:lb}), we finally get
\begin{equation*}
    \begin{aligned}
        \varphi(x)-\varphi(y^+) & \geq \varphi(x)-h(y^+,y)-(\frac{L}{2}-\frac{1}{2t})\| y^+-y\|^2                                                                                                                       \\
                                & \geq -\frac{1}{2t}\| y^+-y\|^2+\left\langle x-y^+,\nabla f(y)+y_s\right\rangle-(\frac{L}{2}-\frac{1}{2t})\| y^+-y\|^2 +\frac{\mu}{2}\|x-y\|^2                                         \\
                                & \mathop{=}\limits^{\text{\eqref{eq:zero}}} -\frac{1}{2t}\| y^+-y\|^2+\frac{1}{t}\left\langle  y-y^+,x-y^+\right\rangle-(\frac{L}{2}-\frac{1}{2t})\| y^+-y\|^2+\frac{\mu}{2}\|x-y\|^2.
    \end{aligned}
\end{equation*}
The desired conclusion follows  by substituting $\mathcal{G}(y,t)=t^{-1}(y-y^+)$ into the above relationship. This completes the proof.

\small
\bibliographystyle{spmpsci} %set the bib style as IEEETran, siamplain or plain
%\bibliography{EEB}
%\bibliography{references}

\end{document}